\documentclass[11pt]{article}

\usepackage{amssymb}
\usepackage{mathtools}
\usepackage[numbers]{natbib}
\usepackage[margin=1in]{geometry}
\usepackage[hidelinks]{hyperref}

\usepackage{titlesec}
\titleformat{\section}
{\centering\large\bfseries}{\thesection.}{.5em}{}
\titleformat{\subsection}
{\large\bfseries}{\thesubsection.}{.5em}{}
\titleformat{\subsubsection}
{\large\bfseries}{\thesubsubsection.}{.5em}{}

\usepackage{amsthm}
\newtheorem{thm}{Theorem}[section]
\newtheorem{lem}[thm]{Lemma}
\newtheorem{prop}[thm]{Proposition}
\newtheorem{cor}[thm]{Corollary}

\theoremstyle{definition}

\usepackage{stucky}

\usepackage{chngcntr}
\counterwithin*{equation}{section}

\begin{document}
\bibliographystyle{abbrv}

\begin{center}
{\large\bf THE DISTRIBUTION OF $\gcd(n,\phi(n))$}\\[1.5em]
{\scshape Joshua Stucky}
\end{center}
\vspace{.5em}

\begin{abstract}
Let $\phi(n)$ denote Euler's phi function. In this paper, we study the distribution of the numbers $(n,\phi(n))$ and their divisors. Our results generalize previous results of Erd\H{o}s \cite{ErdosGroup} and Pollack \cite{PollackOrders}.
\end{abstract}


%

\section{Introduction}

Let $\phi(n)$ denote Euler's phi function. Szele \cite{Szele} showed that the condition $(n,\phi(n)) = 1$ is equivalent to the uniqueness of a group of order $n$, and Erd\H{o}s \cite{ErdosGroup} subsequently showed that
\begin{equation}\label{eq:ErdosGroup}
\sum_{\substack{n\leq x\\ (n,\phi(n))=1}} 1 \sim \frac{x}{e^\gamma\log_3 x}.
\end{equation}
Here and throughout, $\log x = \log_1 x$ is the natural logarithm, $\log_k x = \log(\log_{k-1}(x))$ is the $k$-fold iterated natural logarithm, and $\gamma$ denotes the Euler-Mascheroni constant. Pollack \cite{PollackOrders} has given an asymptotic expansion for the sum \eqref{eq:ErdosGroup} of the form
\begin{equation}\label{eq:PollackPoincare}
\sum_{\substack{n\leq x\\ (n,\phi(n))=1}} 1 = \frac{x}{e^\gamma\log_3 x}\sumpth{1+\sum_{k=1}^K \frac{c_k}{(\log_3 x)^k}} + O_K\fracp{x}{(\log_3 x)^{K+2}}
\end{equation}
for any integer $K\geq 1$, where $c_k$ are certain arithmetic constants. Note that by M\"{o}bius inversion,
\[
\sum_{\substack{n\leq x\\ (n,\phi(n))=1}} 1 = \sum_{d\leq x} \mu(d) \cala_d(x),
\]
where $\cala_d(x)$ denotes the congruence sum
\[
\cala_d(x) = \sum_{\substack{n\leq x\\ d\mid (n,\phi(n))}} 1.
\]
In this paper, we consider the more general sums
\[
S_g(x) : = \sum_{d\leq x} g(d) \cala_d(x),
\]
where $g$ is a multiplicative function. One could consider a variety of arithmetic functions $g$ in the sum above. For clarity of presentation, we restrict our attention to bounded $g$. This is a fairly natural class of arithmetic functions that already allows us to give some interesting applications. Letting $f(n) = \sum_{d\mid n} g(d)$, we have
\[
S_g(x) = \sum_{n\leq x} f((n,\phi(n))).
\]
Our main theorem, Theorem \ref{thm:Main}, furnishes an asymptotic for the sum $S_g$. When $f$ (or, equivalently, $g$) is suitably regular over primes, we obtain an asymptotic expansion of $S_g(x)$ with a level of precision comparable to \eqref{eq:PollackPoincare}. By ``suitably regular'', we mean that we can evaluate the partial sums $\sum_{p\leq x} g(p)p^{-1}$ with a level of precision akin to the Prime Number Theorem. For all the specific $g$ we consider, $g$ will be essentially constant on primes in the sense that the following estimate holds: there are constants $x_0,C_1,C_2$ such that for any $A > 0$ and all $x > x_0$, we have
\[
\sum_{p\leq x} \frac{g(p)}{p} = C_1 \log_2 x + C_2 + O_A\fracp{1}{(\log x)^A}.
\]

Note that if $f$ is multiplicative with $f(n) \in [0,1]$ for all $n$, then by M\"{o}bius inversion, $\abs{g(n)} \leq 1$. In particular, Theorem \ref{thm:Main} allows us to obtain information about the distribution of $(n,\phi(n))$ in any set $\calb$ whose indicator function is multiplicative. Examples of such sets include
\begin{itemize}
\item sums of two squares (counted without multiplicity of representations)
\item $r$th powers (e.g. squares, cubes, etc.)
\item $r$-free numbers (e.g. squarefree or cubefree integers)
\item smooth numbers (i.e. integers free of large prime factors)
\item rough numbers (i.e. integers free of small prime factors)
\end{itemize}
If we instead take $g$ itself to be the indicator function of a set $\calb$ such that $g$ is multiplicative, we obtain information about the distribution of the divisors of $(n,\phi(n))$ that belong to $\calb$. In particular, if $\calb = \N$ (i.e. $g(n) = 1$), then we are able to calculate the average order of the divisor function $\tau$ over $(n,\phi(n))$ (in fact, this was the original motivation for the present work). In Section \ref{sec:PrimeSums}, we give a variety of applications of Theorem \ref{thm:Main}.\\

To further motivate our results, consider the following heuristic argument. Let $P(d)$ denote the largest prime factor of $d$. When $n$ is of size $x$, the condition $d\mid (n,\phi(n))$ is roughly equivalent to $d\mid n$ and $P(d) \leq \log_2 x$. This phenomenon is apparent in Erd\H{o}s' \cite{ErdosGroup} proof of \eqref{eq:ErdosGroup} and has been described more precisely by Erd\H{o}s, Luca, and Pomerance (see Theorem 8 of  \cite{ErdosLucaPomerance}). Heuristically, then, we expect 
\begin{equation}\label{eq:ConjectureSum}
S_g(x) \sim x\sum_{\substack{d\leq x\\ P(d)\leq \log_2 x}} \frac{g(d)}{d}.
\end{equation}
When $g$ is multiplicative and does not grow too quickly, we expect the sum to be well-approximated by the corresponding product over primes, and so we conjecture that
\begin{equation}\label{eq:ConjectureProduct}
S_g(x) \sim x \prod_{p\leq \log_2 x} \sumpth{\sum_{j\geq 0} \frac{g(p^j)}{p^j}}.
\end{equation}
Our main theorem establishes this conjecture in a quantitative form for bounded multiplicative functions. 

\begin{thm}\label{thm:Main}
Let $g$ be a multiplicative function such that $\abs{g(n)} \leq G$ for all $n$ and some constant $G > 0$. For any $K\geq 1$, as $x\to\infty$, we have
\[
S_g(x) = x \sumpth{\prod_{p\leq \log_2 x} \sum_{j\geq 0} \frac{g(p^j)}{p^j}} \exp\pth{\calq_g(x)} +O_{G,K}\fracp{x}{(\log_3 x)^{K}}.
\]
where
\[
\calq_g(x) = -\sum_{p\leq \log_2 x} \frac{g(p)}{p(\log x)^{1/p}} + \sum_{p > \log_2 x} \frac{g(p)}{p} \pth{1-\frac{1}{(\log x)^{1/p}}}.
\]
\end{thm}

Since $g$ is bounded, we immediately obtain the following corollary to Theorem \ref{thm:Main}.

\begin{cor}\label{cor:asymptotic}
Let the hypotheses be as in Theorem \ref{thm:Main}. Then
\[
S_g(x) = x \sumpth{\prod_{p\leq \log_2 x} \sum_{j\geq 0} \frac{g(p^j)}{p^j}} \pth{1+O_{G}\fracp{1}{\log_3 x}} + O_{G,K}\fracp{x}{(\log_3 x)^{K}}.
\]
\end{cor}

\begin{proof}
Note that $t\mapsto t(\log x)^{1/t}$ is decreasing for $t\leq \log_2 x$. As well, since $1-e^{-t} \leq t$ for all $t\geq 0$, we have
\[
1-\frac{1}{(\log x)^{1/p}} \leq \frac{\log_2 x}{p}.
\]
Thus
\[
\abs{\calq_g(x)} \leq G\frac{\pi(\log_2 x)}{e\log_2 x} + G\sum_{p > \log_2 x} \frac{\log_2 x}{p^2} \ll_G \frac{1}{\log_3 x}.
\]
\end{proof}

\noindent\textbf{Remark.} It is possible for the product in Theorem \ref{thm:Main} to be zero, as can be seen by choosing any multiplicative function $g$ such that $g(2^a) = -1$ for $a\geq 1$. In this case, Theorem \ref{thm:Main} only gives an upper bound for $S_g$, rather than an asymptotic. The author thanks Greg Martin for this observation.\\

\noindent\textbf{Outline of the Paper.} In Sections \ref{sec:Prelims} and \ref{sec:Tech}, we collect together a number of preliminary results that we will need in our proof. Section \ref{sec:Tech} contains the main technical tool used in the proof of Theorem \ref{thm:Main}, namely Corollary \ref{cor:NoDivide}. We prove Theorem \ref{thm:Main} in Sections \ref{sec:Cleaning} -- \ref{sec:Conclusion}. We begin by cleaning and decomposing the sum $S_g(x)$ in Section \ref{sec:Cleaning}, removing numerous inconvenient terms and isolating the parts of the summation variables that govern the behavior of $S_g(x)$. After these manipulations, we apply Corollary \ref{cor:NoDivide} in Section \ref{sec:FinalEval} and perform some final technical manipulations in Section \ref{sec:Conclusion} to complete the proof of Theorem \ref{thm:Main}. Finally, in Section \ref{sec:PrimeSums}, we show how to evaluate the prime sums $\calq_g(x)$ Theorem \ref{thm:Main} for a variety of arithmetic functions $g$, obtaining asymptotic expansions of $S_g(x)$ with a level of precision analogous to \eqref{eq:PollackPoincare}.\\

\noindent\textbf{Acknowledgments.} The author would like to thank Paul Pollack for useful discussions in the development of these results, and in particular for suggesting the argument used in Section \ref{sec:FinalTruncation}.\\

\noindent\textbf{Notation.} In addition to the notation already discussed, we make the following conventions. The letters $p$ and $q$ always denote primes numbers. The functions $P(n)$ and $p(n)$ are the largest and smallest prime factors of $n$, respectively, and we set $P(1) = 1$ and $p(1) = \infty$. The symbol $\sumflat$ denotes a sum over squarefree integers. Throughout, we use $C$ to denote a positive absolute constant coming from the error term in the prime number theorem. To keep our notation simple, we may modify the value of $C$ at each occurrence. Thus we may write, for example,
\[
\log_3 x \exp(-C\sqrt{\log_3 x}) \ll \exp(-C\sqrt{\log_3 x}).
\]
Rather than writing subscripts throughout, we allow all implied constants to depend on the constants $G$ and $K$ defined in Theorem \ref{thm:Main}.

\section{Preliminaries}\label{sec:Prelims}

In this section, we collect together some preliminary results we will need in our analysis. We begin with the following simple lemma which gives a useful divisibility property for $\phi$.

\begin{lem}\label{lem:PhiProps}${}$
For any positive integers $a$ and $b$, we have $\phi(a) \mid \phi(ab)$. In particular, if $d$ is a positive integer with $d\nmid \phi(ab)$, then $d\nmid \phi(a)$.
\end{lem}

\begin{proof}
We write $b= b_a b'$, where $b_a$ is the largest divisor of $b$ supported on the primes dividing $a$. Then
\[
\phi(ab) = \phi(ab_a)\phi(b') = b_a \phi(a) \phi(b')
\]
and the claim follows.
\end{proof}

We will need the following forms of Mertens' theorems. These are straightforward consequences of the prime number theorem with classical error term $O(x\exp(-C\sqrt{\log x}))$.

\begin{lem}\label{lem:Mertens}
There is an absolute constant $c$ such that, for all $x \geq 3$, 
\begin{equation}\label{eq:MertensSum}
\sum_{p\leq x} \frac{1}{p} = \log_2 x + c + O(\exp(-C\sqrt{\log x})).
\end{equation}
We also have
\begin{equation}\label{eq:MertensProduct}
\prod_{p\leq x} \pth{1-\frac{1}{p}}^{-1} = e^\gamma \log x + O(\exp(-C\sqrt{\log x})).
\end{equation}
\end{lem}

Next, we need Rankin's trick in the following form.

\begin{lem}\label{lem:RankinsTrickProduct}
For any $x \geq y \geq 10$, we have
\begin{equation}\label{eq:Rankin2}
\sum_{\substack{d > x\\ P(d) \leq y}} \frac{1}{d} \ll \frac{(\log y)^3}{x^{\frac{1}{\log y}}}.
\end{equation}
\end{lem}

\begin{proof}
Let $\sigma = \frac{1}{\log y}$. Then
\[
\sum_{\substack{d > x\\ P(d) \leq y}} \frac{1}{d} \leq \frac{1}{x^\sigma} \sum_{P(d) \leq y} \frac{1}{d^{1-\sigma}} = \frac{1}{x^{\sigma}} \prod_{p\leq y} \pth{1-\frac{1}{p^{1-\sigma}}}^{-1}.
\]
Since $y\geq 10$, \eqref{eq:MertensProduct} yields
\[
\prod_{p\leq y} \pth{1-\frac{1}{p^{1-\sigma}}}^{-1} \leq 20 \prod_{10 < p\leq y} \pth{1-\frac{e}{p}}^{-1} \leq \prod_{10 < p\leq y} \pth{1-\frac{1}{p}}^{-3} \ll (\log y)^3.
\]
\end{proof}

%

In the course of proving our results (specifically Proposition \ref{prop:FundamentalNoDivide}), we will need to use the Fundamental Lemma of Sieve Theory. Of the versions available, we will use the one due to Koukoulopolus (see \cite{Koukoulopolus}, Theorem 18.11). The following lemma, which follows from the fundamental lemma, is stated without proof in a slightly different form in \cite{LPR}. For completeness, we include a short proof.

\begin{lem}\label{lem:FundamentalLemmaST}
Let $X$ and $Z$ be such that $2\leq Z\leq X^{1/2e^2}$, and let $\calq$ be a set of primes not exceeding $Z$. Let $(u,v]$ be an interval with $v-u = X$. For each $q\in \calq$, choose a residue class $a_q \mod{q}$. Let $N(u,v;\calq)$ be the number of integers $n \in (u,v]$ for which $n \not\equiv a_q \mod{q}$ for any $q\in\calq$. Then
\[
N(u,v;\calq) = X \prod_{q\in\calq} \pth{1-\frac{1}{q}}\pth{1+O\pth{\exp\pth{-\frac{\log X}{\log Z}}}}.
\]
\end{lem}

\begin{proof}
By the Chinese Remainder Theorem, we may choose an integer $a_0$ such that $a_q \equiv a_0 \mod{q}$ for each $q\in \calq$. We then sieve the sequence $n-a_0$ since, for $q\in\calq$, $q\mid n-a_0$ if and only if $n\equiv a_0 \equiv a_q \mod{q}$. Thus
\[
N(u,v;\calq) = \sum_{\substack{u < n\leq v\\ (n-a_0,\calq)=1}} 1 = \sum_{\substack{u-a_0 < n\leq v-a_0\\ (n,\calq)=1}} 1,
\]
where the condition $(n,\calq)=1$ means that $n$ is free of primes in $\calq$. To apply Theorem 18.11 of \cite{Koukoulopolus}, we note that for any $d$,
\[
\sum_{\substack{u-a_0 < n\leq v-a_0\\ n\equiv 0 \mod{d}}} 1 = \frac{X}{d} + O(1).
\]
Moreover, for any $2 \leq y_1 < y_2 \leq Z$, we have
\[
\prod_{\substack{p\in \calq \cap (y_1,y_2]}} \pth{1-\frac{1}{p}}^{-1} \leq \prod_{y_1 < p \leq y_2} \pth{1-\frac{1}{p}}^{-1} = \frac{\log y_2}{\log y_1}\pth{1+O\fracp{1}{\log y_1}}
\]
by Lemma \ref{lem:Mertens}. Thus Axioms 1 and 2 of Theorem 18.11 of \cite{Koukoulopolus} are satisfied. We take $u = \frac{\log X}{2\log Z}$ in the theorem and note that $u^{-u/2} \leq e^{-u}$ so long as $\log u \geq 2$. This is equivalent to $Z \leq X^{1/2e^2}$.

\end{proof}

To understand the product over primes appearing in Lemma \ref{lem:FundamentalLemmaST}, we need the following result of Pomerance \cite{Pomerance}.

\begin{lem}\label{lem:Pomerance}
Let $a,m$ be positive integers with $(a,m)=1$ and let $p_{a,m}$ denote the least prime $p\equiv a \mod{m}$. Then
\[
\sum_{\substack{p\leq x\\ p\equiv a\mod{m}}} \frac{1}{p} = \frac{\log_2 x}{\phi(m)} + \frac{1}{p_{a,m}} + O\fracp{\log(2m)}{\phi(m)}.
\]
In particular, 
\[
\sum_{\substack{p\leq x\\ p\equiv 1\mod{m}}} \frac{1}{p} = \frac{\log_2 x}{\phi(m)} + O\fracp{\log(2m)}{\phi(m)}.
\]
\end{lem}

\section{Main Technical Propositions}\label{sec:Tech}

In this section, we prove several results that we will use to evaluate sums involving the divisibility conditions $d\mid\phi(n)$ and $d\nmid\phi(n)$. The first is a simple consequence of the Brun-Titchmarsh inequality.

\begin{prop}\label{prop:FundamentalDivide}
For any $x > 1$ and $p < x$, we have
\[
\sum_{\substack{n\leq x\\ p\mid\phi(n)}} 1 \ll \frac{x\log_2 x}{p}.
\]
\end{prop}

\begin{proof}
If $p\mid \phi(n)$, then either $p^2\mid n$ or $n$ is divisible by a prime $q\equiv 1 \mod{p}$. The Brun-Titchmarsh inequality and partial summation then gives
\[
\sum_{\substack{n\leq x\\ p\mid\phi(n)}} 1 \leq \frac{x}{p^2} + \sum_{\substack{q\leq x\\ q\equiv 1 \mod{p}}} \frac{x}{q} \ll \frac{x\log_2 x}{p}.
\]
\end{proof}

The following proposition allows us to evaluate asymptotically sums involving divisibility conditions on $\phi(n)$.

\begin{prop}\label{prop:FundamentalNoDivide}
Let $x$ be sufficiently large. Denote by $\alpha$ and $\beta$ the functions
\begin{equation}\label{eq:AlphaBetaDefs}
\alpha(d) = \prod_{p\mid d}\pth{1-\frac{1}{p-1}}, \qquad \beta(d) = \sum_{p\mid d} \frac{\log p}{p}.
\end{equation}
Let $d$ be a positive squarefree integer such that the following hold:
\begin{enumerate}[label={\normalfont(\roman*)}]
\item $d$ is odd and $1<d \leq \exp\pth{\frac{1}{2}\sqrt{\log x}}$;
\item $\frac{\beta(d)d}{\phi(d)}$ is bounded.
\end{enumerate}
Then
\begin{equation}\label{eq:FundamentalNoDivideAsymp}
\sum_{\substack{n\leq x\\ (\phi(n),d)=1}} 1 = \frac{x}{(\log x)^{1-\alpha(d)}}\pth{1+O\fracp{\beta(d)d}{\phi(d)}}.
\end{equation}
If $\omega(d) \leq K$ for some $K \geq 1$, then
\begin{equation}\label{eq:BoundedBound}
\frac{\beta(d)d}{\phi(d)} \ll_K \frac{\log p(d)}{p(d)}.
\end{equation}
\end{prop}

\begin{proof}
Let
\[
\calq = \bigcup_{p\mid d} \set{q: q\equiv 1 \mod{p}}, \qquad Q(z) = \prod_{\substack{q\in \calq\\ q\leq z}} q
\]
and note that by the Chinese Remainder Theorem and Lemma \ref{lem:Pomerance},
\[
\begin{aligned}
\sum_{\substack{q\leq x\\ q\in \calq}} \frac{1}{q} &= -\sum_{\substack{l\mid d\\ l > 1}} \mu(l) \sum_{\substack{q\leq x\\ q\equiv 1 \mod{l}}} \frac{1}{q} \\
&= -\sum_{\substack{l\mid d\\ l > 1}} \mu(l) \sumpth{\frac{\log_2 x}{\phi(l)} + O\fracp{\log l}{\phi(l)}} \\
&= (1-\alpha(d))\log_2 x + O\sumpth{\sum_{l\mid d}\frac{\log l}{\phi(l)}}.
\end{aligned}
\]
To estimate the error term, let
\[
g_d(s) = \sum_{l\mid d} \frac{1}{\phi(l)^s}.
\]
By the well-known inequalities
\[
\frac{l}{\log_2 l} \ll \phi(l) \leq l,
\]
we have $\log l \asymp \log\phi(l)$, and so the error term above is
\[
\asymp \sum_{l\mid d} \frac{\log\phi(l)}{\phi(l)} = -g_d'(1).
\]
But for $d$ squarefree,
\[
g_d(s) = \prod_{p \mid d} \pth{1+\frac{1}{(p-1)^s}} = \exp\sumpth{\sum_{p\mid d} \log\pth{1+\frac{1}{(p-1)^s}}},
\]
and so
\[
-g_d'(1) = \sumpth{\sum_{p \mid d} \frac{\log(p-1)}{p}} \prod_{p\mid d} \pth{1+\frac{1}{p-1}} = \frac{d}{\phi(d)}\sum_{p \mid d} \frac{\log(p-1)}{p} <  \frac{\beta(d)d}{\phi(d)}.
\]
Thus
\begin{equation}\label{eq:PrimeReciprocalsTheta}
\sum_{\substack{q\leq x\\ q\in \calq}} \frac{1}{q} = (1-\alpha(d))\log_2 x + O\fracp{\beta(d)d}{\phi(d)}.
\end{equation}

Let $J \geq 2e^2$ be a parameter to be chosen. For the upper bound in \eqref{eq:FundamentalNoDivideAsymp}, we note that the condition $(\phi(n),d)=1$ implies that $n$ is free of primes $q \equiv 1 \mod{p}$ for each prime $p\mid d$. In particular, $n$ is free of all such primes $q\leq x^{1/J}$. By Lemma \ref{lem:FundamentalLemmaST} with $z = x^{1/J}$, the number of integers up to $x$ free of primes dividing $Q(z)$ is
\[
x \prod_{\substack{q\leq x^{1/J}\\ q\in \calq}} \pth{1-\frac{1}{q}} \pth{1+O\pth{e^{-J}}}.
\]
To evaluate the product over primes, we first write
\[
\prod_{\substack{q\leq x^{1/J}\\ q\in \calq}} \pth{1-\frac{1}{q}} = \exp\sumpth{\sum_{\substack{q\leq x^{1/J}\\ q\in \calq}} \log\pth{1-\frac{1}{q}}} = \exp\sumpth{-\sum_{\substack{q\leq x^{1/J}\\ q\in \calq}} \frac{1}{q} + O\sumpth{\sum_{\substack{q\leq x^{1/J}\\ q\in \calq}}\frac{1}{q^2}}}.
\]
The error term is
\[
\sum_{\substack{q\leq x^{1/J}\\ q\in \calq}}\frac{1}{q^2} \leq \sum_{p\mid d} \sum_{q\equiv 1\mod{p}}\frac{1}{q^2} \ll \sum_{p\mid d} \frac{1}{p^2} \ll \frac{1}{p(d)},
\]
and so
\[
\prod_{\substack{q\leq x^{1/J}\\ q\in \calq}} \pth{1-\frac{1}{q}} =\exp\sumpth{-\sum_{\substack{q\leq x^{1/J}\\ q\in \calq}} \frac{1}{q}}\pth{1+O\fracp{1}{p(d)}}.
\]
By \eqref{eq:PrimeReciprocalsTheta}, the remaining sum over $q$ is
\[ 
\sum_{\substack{q\leq x^{1/J}\\ q\in \calq}} \frac{1}{q} = (1-\alpha(d))(\log_2 x-\log J) + O\fracp{\beta(d)d}{\phi(d)}.
\]
and thus
\begin{equation}\label{eq:PrimeProductExponential}
\prod_{\substack{q\leq x^{1/J}\\ q\in \calq}} \pth{1-\frac{1}{q}}
= \frac{J^{1-\alpha(d)}}{(\log x)^{1-\alpha(d)}}\pth{1+O\fracp{\beta(d)d}{\phi(d)}},
\end{equation}
since 
\[
\frac{\beta(d)d}{\phi(d)} \geq \beta(d) \geq \frac{1}{p(d)}.
\]
for $p(d) \geq 3$. Therefore
\[
\sum_{\substack{n\leq x\\ (\phi(n),d)=1}} 1 \leq \frac{J^{1-\alpha(d)}x}{(\log x)^{1-\alpha(d)}}\pth{1+O\pth{e^{-J} + \frac{\beta(d)d}{\phi(d)}}}.
\]

For the lower bound, note that if $(n,d) = 1$ and $n$ is free of primes in $\calq$, then $(\phi(n),d)=1$. Thus
\begin{equation}\label{eq:Diff}
\begin{aligned}
\sum_{\substack{n\leq x\\ (\phi(n),d)=1}} 1 \geq \sum_{\substack{n\leq x\\ (n,dQ(x))=1}} 1 &=  \sum_{\substack{n\leq x\\ (n,dQ(x^{1/J}))=1}} 1 - \sum_{\substack{n\leq x\\ (n,Q(x)) > 1\\ (n,dQ(x^{1/J}))=1}} 1 \\
&\geq  \sum_{\substack{n\leq x\\ (n,dQ(x^{1/J}))=1}} 1 - \sum_{\substack{n\leq x\\ (n,Q(x)) > 1\\ (n,Q(x^{1/J}))=1}} 1.
\end{aligned}
\end{equation}
Reasoning as before, the first sum is
\[
\frac{\phi(d)}{d}\frac{J^{1-\alpha(d)}x}{(\log x)^{1-\alpha(d)}}\pth{1+O\pth{e^{-J} + \frac{\beta(d)d}{\phi(d)}}}.
\]
The hypotheses of the proposition imply that $\beta(d)$ is bounded, and so
\[
\frac{\phi(d)}{d} = \exp\sumpth{\sum_{p\mid d} \log\pth{1-\frac{1}{p}}} = \exp(O(\beta(d))) =  1 + O(\beta(d)),
\]
and so
\[
\sum_{\substack{n\leq x\\ (n,dQ(x^{1/J}))=1}} 1 = \frac{J^{1-\alpha(d)}x}{(\log x)^{1-\alpha(d)}}\pth{1+O\pth{e^{-J} + \frac{\beta(d)d}{\phi(d)}}}.
\]
To bound the second sum in \eqref{eq:Diff}, note that the conditions $(n,Q(x)) > 1$, $(n,Q(x^{1/J}))=1$ induce a factorization $n=AB$, where $A$ is the largest divisor of $n$ supported on primes dividing $Q(x)/Q(x^{1/J})$. Every such $A$ satisfies $A \equiv 1 \mod{l}$ for some $l\mid d$, $l > 1$. We fix $B$ and $l$ and count the number of corresponding $A$. Note that $1 < A \leq x/B$ and that every prime dividing $A$ exceeds $x^{1/J}$. As well, $A = 1 +la$ for some $a < x/lB$.  If we assume for the moment that $\frac{\log x}{\log l} \geq 2J$, then $x/lB > x^{1/2J}$ since $B < x^{1-1/J}$. For any prime $q\leq x^{1/2J}$ with $q\nmid l$, we have $a \equiv \bar{l}(A-1) \mod{q}$. Since every prime that divides $A$ is at least $x^{1/J}$, we have $A \not\equiv 0 \mod{q}$, and so $a\not\equiv -\bar{l} \mod{q}$. Applying Lemma \ref{lem:FundamentalLemmaST} to $a$ with $a_q \equiv -\bar{l} \mod{q}$, the number of $A$ corresponding to a given $B$ and $l$ is 
\[
\ll \frac{x}{lB} \prod_{\substack{q\leq x^{1/2J}\\ q\nmid l}} \pth{1-\frac{1}{q}} \leq \frac{x}{lB} \prod_{q\leq x^{1/2J}} \pth{1-\frac{1}{q}} \prod_{q\mid l} \pth{1-\frac{1}{q}}^{-1} \ll \frac{xJ}{\phi(l)B\log x}.
\]
Summing over $l$, the number of $A$ corresponding to a given $B$ is
\[
\ll \frac{xJ}{B\log x} \sum_{\substack{l\mid d\\ l > 1}} \frac{1}{\phi(l)}.
\]
Mertens' theorem then gives
\[
\begin{aligned}
\sum \frac{1}{B} \leq \prod_{\substack{q\leq x\\ q\not\in\calq}} \pth{1-\frac{1}{q}}^{-1} &= \prod_{\substack{q\leq x}} \pth{1-\frac{1}{q}}^{-1} \prod_{\substack{q\leq x\\ q\in\calq}} \pth{1-\frac{1}{q}} \\
&\ll \log x\exp\sumpth{-\sum_{\substack{q\leq x\\ q\in\calq}} \frac{1}{q}} \\
&\ll (\log x)^{\alpha(d)}
\end{aligned}
\]
Therefore
\[
\sum_{\substack{n\leq x\\ (n,Q(x)) > 1\\ (n,Q(x^{1/J}))=1}} 1 \ll \frac{xJ}{(\log x)^{1-\alpha(d)}} \sum_{\substack{l\mid d\\ l > 1}} \frac{1}{\phi(l)}.
\]
Letting $\theta(d)$ denote the sum over $l$, we have shown
\[
\sum_{\substack{n\leq x\\ (\phi(n),d)=1}} 1 = \frac{J^{1-\alpha(d)}x}{(\log x)^{1-\alpha(d)}}\pth{1+O\pth{e^{-J} + \frac{\beta(d)d}{\phi(d)} + J^{\alpha(d)} \theta(d)}}.
\]
Since $p(d) \geq 3$, we have
\[
\alpha(d) = \prod_{p\mid d}\pth{1-\frac{1}{p-1}} \geq \exp\sumpth{-\sum_{p\mid d} \frac{2}{p-1}} \geq \exp\pth{-\frac{4\beta(d)}{\log p(d)}}\geq 1-\frac{4\beta(d)}{\log p(d)},
\]
and also
\[
\begin{aligned}
1+\theta(d) = \prod_{p\mid d} \pth{1+\frac{1}{p-1}}  \leq \exp\sumpth{\sum_{p\mid d} \frac{1}{p-1}} &\leq \exp\fracp{2\beta(d)}{\log p(d)} \\
&= 1 + O\fracp{\beta(d)}{\log p(d)}.
\end{aligned}
\]
That is, both $1-\alpha(d)$ and $\theta(d)$ are $O\fracp{\beta(d)}{\log p(d)}$. If $p(d) \geq e^{2e^2}$, we choose $J = \log p(d)$ so that the hypotheses of the proposition ensure that $\frac{\log x}{\log l} \geq 2J \geq 4e^2$ for every $l\mid d$, $l > 1$ and $x$ sufficiently large. Since $\alpha(d) \leq 1$, we then have $J^{\alpha(d)}\theta(d) \ll \beta(d)$ and 
\[
J^{1-\alpha(d)} = \exp((1-\alpha(d)) \log_2 p(d)) = 1 + O(\beta(d)).
\]
Therefore
\[
\sum_{\substack{n\leq x\\ (\phi(n),d)=1}} 1 = \frac{x}{(\log x)^{1-\alpha(d)}}\pth{1+O\fracp{\beta(d)d}{\phi(d)}},
\]
if $p(d) \geq e^{2e^2}$. If $p(d) < e^{2e^2}$, the $\beta(d) \gg 1$, and so \eqref{eq:FundamentalNoDivideAsymp} can only be an upper bound. In this case, the hypotheses of the proposition and \eqref{eq:PrimeReciprocalsTheta} imply that there is an absolute constant $C > 0$ such that
\[
\sum_{\substack{q\leq x\\ q\in \calq}} \frac{1}{q} \geq (1-\alpha(d))\log_2 x -C.
\]
The desired upper bound then follows from Brun's sieve, specifically in the form of Theorem 2.3 of \cite{HalberstamRichert}.

To prove \eqref{eq:BoundedBound}, we simply note that if $\omega(d) \leq K$, then
\[
\sum_{p\mid d} \frac{\log p}{p} \ll K \frac{\log p(d)}{p(d)}.
\]
and
\[
\frac{d}{\phi(d)} = \prod_{p\mid d} \pth{1-\frac{1}{p}}^{-1} \leq \prod_{r=1}^K \pth{1-\frac{1}{p_r}}^{-1},
\]
where $p_r$ denotes the $r$th prime. We have $p_r \ll r\log r$, and so \eqref{eq:MertensProduct} gives
\[
\prod_{r=1}^K \pth{1-\frac{1}{p_r}}^{-1} \ll \log p_K \ll \log K.
\]

\end{proof}

The asymptotic formula \eqref{eq:FundamentalNoDivideAsymp} is not quite suitable for our purposes. To put it into a form more amenable for calculations, we need the following estimate for the function $\alpha(d)$ appearing in Proposition \ref{prop:FundamentalNoDivide}.

\begin{lem}\label{lem:alphaProp}
Let
\[
\alpha_1(d) = \sum_{p\mid d} \frac{1}{p-1}.
\]
If $K \geq 1$ and $\omega(d) \leq K$, then
\[
\abs{1-\alpha(d)-\alpha_1(d)} \leq \frac{4e^K}{p(d)^2}.
\]
\end{lem}

\begin{proof}
Expanding the product in the definition of $\alpha(d)$ and grouping the terms depending on the value of $\omega(g)$, we have
\[
1-\alpha(d) = - \sum_{\substack{g\mid d\\ g>1}} \frac{\mu(g)}{\phi(g)} = \alpha_1(d) - \sum_{k=2}^{\omega(d)} (-1)^k \sum_{\substack{g\mid d\\ \omega(g) = k}} \frac{1}{\phi(g)}.
\]
The triangle inequality and binomial theorem then imply
\[
\begin{aligned}
\abs{1-\alpha(d)-\alpha_1(d)} &\leq \sum_{k=2}^{\omega(d)} \frac{1}{k!} \sumpth{\sum_{p\mid d} \frac{1}{p-1}}^k \leq \sum_{k=2}^{\omega(d)} \frac{1}{k!} \fracp{\omega(d)}{p(d)-1}^k \leq \frac{4}{p(d)^2} \sum_{k\geq 2} \frac{K^k}{k!} \leq \frac{4e^K}{p(d)^2}.
\end{aligned}
\]
\end{proof}

Combining this with \eqref{eq:FundamentalNoDivideAsymp}, we deduce the following consequences of Proposition \ref{prop:FundamentalNoDivide}. These are the forms of Proposition \ref{prop:FundamentalDivide} that we will use most often in applications.

\begin{cor}\label{cor:NoDivide}
Let the notation and hypotheses be as in Proposition \ref{prop:FundamentalNoDivide} and let $\alpha_1$ be as in Lemma \ref{lem:alphaProp}. If $K\geq 1$ and $\omega(d) \leq K$, and if $p(d) > 2\sqrt{e^K\log_2 x}$, then
\begin{equation}\label{eq:FundamentalNoDivideSimpler}
\sum_{\substack{n\leq x\\ (\phi(n),d)=1}} 1 = \frac{x}{(\log x)^{\alpha_1(d)}}\pth{1+O_K\pth{\frac{\log p(d)}{p(d)} + \frac{\log_2 x}{p(d)^2}}},
\end{equation}
We also have
\begin{equation}\label{eq:FundamentalDivide}
\sum_{\substack{n\leq x\\ d\mid \phi(n)}} 1 = x \prod_{p\mid d} \pth{1-\frac{1}{(\log x)^{1/(p-1)}}} + O_K\pth{x\pth{\frac{\log p(d)}{p(d)} + \frac{\log_2 x}{p(d)^2}}}
\end{equation}
and
\begin{equation}\label{eq:FundamentalNoDivideUpper}
\sum_{\substack{n\leq x\\ d\nmid \phi(n)}} 1 \ll x \sum_{p\mid d} (\log x)^{-1/(p-1)}.
\end{equation}
\end{cor}

\begin{proof}
To prove \eqref{eq:FundamentalNoDivideSimpler}, note that by Lemma \ref{lem:alphaProp}, we have, for some $\abs{\theta} \leq 1$,
\[
\begin{aligned}
\frac{1}{(\log x)^{1-\alpha(d)}} = \frac{1}{(\log x)^{\alpha_1(d) + 4e^K\theta p(d)^{-2}}} &= \frac{1}{(\log x)^{\alpha_1(d)}}\exp\fracp{4e^K\theta \log_2 x}{p(d)^2} \\
&= \frac{1}{(\log x)^{\alpha_1(d)}}\pth{1+O_K\fracp{\log_2 x}{p(d)^2}}.
\end{aligned}
\]
The estimate \eqref{eq:FundamentalNoDivideSimpler} now follows by inserting this into \eqref{eq:FundamentalNoDivideAsymp} and applying \eqref{eq:BoundedBound}.

For \eqref{eq:FundamentalDivide}, M\"{o}bius inversion gives
\[
\sum_{\substack{n\leq x\\ d\mid \phi(n)}} 1 = \sum_{g\mid d} \mu(g) \sum_{\substack{n\leq x\\ (g,\phi(n))=1}} 1.
\]
Applying \eqref{eq:FundamentalNoDivideSimpler}, we then have
\[
\begin{aligned}
\sum_{\substack{n\leq x\\ d\mid \phi(n)}} 1 &= \sum_{g\mid d} \mu(g) \sum_{\substack{n\leq x\\ (g,\phi(n))=1}} 1 \\
&= \sum_{\substack{n\leq x}} 1  + \sum_{\substack{g\mid d\\ g>1}} \mu(g) \frac{x}{(\log x)^{\alpha_1(g)}}\pth{1+O_K\pth{\frac{\log p(g)}{p(g)} + \frac{\log_2 x}{p(g)^2}}} \\
&=x \prod_{p\mid d} \pth{1-\frac{1}{(\log x)^{1/(p-1)}}} + O_K\sumpth{x\pth{\frac{\log p(d)}{p(d)} + \frac{\log_2 x}{p(d)^2}}}.
\end{aligned}
\]
To prove \eqref{eq:FundamentalNoDivideUpper}, note that if $d\nmid \phi(n)$, then $p\nmid \phi(n)$ for some $p\mid d$. Thus
\[
\sum_{\substack{n\leq x\\ d\nmid \phi(n)}} 1 \leq \sum_{p\mid d} \sum_{\substack{n\leq x\\ p\nmid \phi(n)}} 1.
\]
Applying \eqref{eq:FundamentalNoDivideSimpler} then gives \eqref{eq:FundamentalNoDivideUpper}.

\end{proof}

\section{Cleaning and Decomposition of $S_g(x)$}\label{sec:Cleaning}

Recall that $g$ is a multiplicative function bounded by 1. For notational simplicity, we suppress the subscript notation and write $S(x) = S_g(x)$. We then have
\[
S(x) = \sum_{d\leq x} g(d) \sum_{\substack{n\leq x\\ d\mid (n,\phi(n))}} 1 = \sum_{d\leq x} g(d) \sum_{\substack{n\leq \frac{x}{d}\\ d\mid \phi(nd)}} 1.
\]
To evaluate this sum, we need to perform some substantial cleaning and decomposition to remove inconvenient terms and isolate the portions of the variables $n$ and $d$ that govern the behavior of the sum. The net result of our manipulations, as well as the total error incurred as a result, can be seen in equation \eqref{eq:FullyCleaned} at the end of this section.

For brevity, let $y = \log_2 x$. As stated in the introduction, the condition $d\mid (n,\phi(n))$ is roughly equivalent to $d\mid n$ and $P(d) \leq y$. Consequently, we expect the behavior of the congruence sums $\cala_d(x)$ to be governed primarily by the prime factors of $d$ that are close to $y$. More precisely, it turns out that this behavior is governed by the prime factors $p$ that are close to $y$ on the logarithmic scale, i.e.
\[
1-\ep \leq \frac{\log p}{\log y} \leq 1+\ep,
\]
where the specific $\ep$ we consider depends on the desired precision (the parameter $K$) of the asymptotic expansion given in Theorem \ref{thm:Main}. Moreover, we expect the prime factors $p$ to affect the congruence sum $\cala_d(x)$ in different ways depending as $p > y$ or $p < y$. This discussion motivates us to decompose the variable $d$ in the following way. Let $U,V \in [1,\sqrt{y}]$ be two parameters to be chosen and write
\begin{equation}\label{eq:dDecomp}
d = s m \ell,
\end{equation}
where
\begin{itemize}
\item $P(s) \leq \frac{y}{V}$,
\item $p(m) > \frac{y}{V}$ and $P(m) \leq Uy$,
\item $p(\ell) > Uy$.
\end{itemize}

\subsection{Initial Truncation}
For technical reasons, it will be convenient to know that $d$ is much smaller than $x$. As such, before introducing the above decomposition of $d$, we first show that the contribution of large $d$ is negligible. Let
\[
D = \exp((\log_3 x)^2), \qquad D' = (\log_2 x)\exp(\sqrt{\log_3 x}),
\]
and let $S^{(1)}$ be the sum $S$ but restricted to $d\leq D$. We have
\[
\abs{S(x) - S^{(1)}(x)} \leq \sum_{\substack{D < d \leq x}} \sum_{\substack{n\leq \frac{x}{d}\\ P(d) \mid \phi(nd)}} 1 = \sumpth{\sum_{\substack{D < d \leq x\\ P(d) \leq D'}} + \sum_{\substack{D < d \leq x\\ P(d) > D'}}} \sum_{\substack{n\leq \frac{x}{d}\\ P(d) \mid \phi(nd)}} 1 = \Sigma_1 + \Sigma_2,
\]
say. To estimate $\Sigma_1$, we use Rankin's trick in the form of \eqref{eq:Rankin2}:
\[
\Sigma_1 = \sum_{\substack{D < d \leq x\\ P(d) \leq D'}} \sum_{\substack{n\leq \frac{x}{d}\\ P(d) \mid \phi(nd)}} 1 \leq \sum_{\substack{d > D\\ P(d) \leq D'}} \frac{x}{d} \ll x \frac{(\log D')^3}{D^{1/\log D'}} \ll \frac{x}{\sqrt{\log_2 x}}.
\]
To estimate $\Sigma_2$, we note that if $P(d) \mid \phi(nd)$, then at least one of the following holds:
\begin{enumerate}
\item $P(d)^2 \mid d$;
\item $P(d) \mid n$;
\item $P(d) \nmid n$ and $P(d) \mid \phi(n)$.
\end{enumerate}
The contribution to $\Sigma_2$ of the first case is
\[
\sum_{\substack{D < d \leq x\\ P(d) > D'\\ P(d)^2\mid d}} \sum_{n\leq \frac{x}{d}} 1 \leq \sum_{p > D'} \frac{x}{p^2} \sum_{\substack{d > D/p^2\\ P(d) \leq p}} \frac{1}{d} \ll x\sum_{p > D'} \frac{\log p}{p^2} \ll \frac{x}{D'}
\]
and the contribution of the second case satisfies the same estimate. For the third case, we use Proposition \ref{prop:FundamentalDivide} to see that this contribution is at most 
\[
\begin{aligned}
\sum_{\substack{D < d \leq x\\ P(d) > D'}} \sum_{\substack{n\leq \frac{x}{d}\\ P(d) \mid \phi(n)}} 1 &\ll \sum_{\substack{D < d \leq x\\ P(d) > D'}} \frac{x\log_2 x}{dP(d)} \ll \sum_{p > D'} \frac{x\log_2 x}{p^2} \sum_{\substack{d > D\\P(d) \leq p}} \frac{1}{d} \ll \frac{x\log_2 x}{D'}.
\end{aligned}
\]
Combining the three estimates above, we have
\[
\Sigma_2 \ll x\exp(-\sqrt{\log_3 x}),
\]
and so
\begin{equation}\label{eq:Cleaning1}
S(x) = S^{(1)}(x) + O\pth{x\exp(-\sqrt{\log_3 x})}
\end{equation}
since $\sqrt{\log_2 x} \gg \exp(\sqrt{\log_3 x})$.

\subsection{Reduction to Primes Close to $y$}
Recall that we expect the behavior of the congruence sums $\cala_d(x)$ to be governed primarily by the prime factors of $d$ that are close to $y$ in the logarithmic scale. To make this precise, we begin by decomposing $S^{(1)}$ using \eqref{eq:dDecomp}. We have
\[
S^{(1)}(x) = \sum_{\substack{s\leq D\\P(s) \leq \frac{y}{V}}} \sum_{\substack{m\leq \frac{D}{s}\\p(m) > \frac{y}{V} \\ P(m) \leq Uy}} \sum_{\substack{\ell\leq \frac{D}{sm}\\ p(\ell) > Uy}} g(sm\ell)\sum_{\substack{n\leq \frac{x}{sm\ell}\\ sm\ell\mid \phi(sm\ell n)}} 1. 
\] 
The contribution of $\ell > 1$ may be estimated in a similar manner as our truncation above: the summation conditions imply that $P(\ell) \mid (n,\phi(n))$, and so arguing as before, the contribution to $S^{(1)}$ of $\ell > 1$ is
\[
\ll \sum_{\substack{s\leq D\\P(s) \leq \frac{y}{V}}} \sum_{\substack{m\leq \frac{D}{s}\\p(m) > \frac{y}{V} \\ P(m) \leq Uy}} \sum_{\substack{1 < \ell\leq \frac{D}{sm}\\ p(\ell) > Uy}} \frac{x\log_2 x}{sm\ell P(\ell)} \ll x\log_2 x \sumpth{\sum_{\substack{d\leq D\\ P(d)\leq Uy}} \frac{1}{d}} \sum_{p > Uy} \frac{1}{p^2} \sum_{\substack{\ell \geq 1\\ p(\ell) > Uy\\ P(\ell) \leq p}} \frac{1}{\ell} \ll \frac{x}{U}.
\]
Thus
\begin{equation}\label{eq:Cleaning2}
S^{(1)}(x) = S^{(2)}(x) + O\pth{x\exp(-\sqrt{\log_3 x}) + \frac{x}{U}},
\end{equation}
where
\[
S^{(2)}(x) = \sum_{\substack{s\leq D\\P(s) \leq \frac{y}{V}}} \sum_{\substack{m\leq \frac{D}{s}\\p(m) > \frac{y}{V} \\ P(m) \leq Uy}} g(sm)\sum_{\substack{n\leq \frac{x}{sm}\\ sm\mid \phi(smn)}} 1. 
\]

Next, we show that we may ignore the condition $s\mid \phi(smn)$ with negligible error. We begin by writing
\[
\sum_{\substack{n\leq \frac{x}{sm}\\ sm\mid \phi(smn)\\ (n,m)=1}} 1 = \sum_{\substack{n\leq \frac{x}{sm}\\ m\mid \phi(smn)\\ (n,m)=1}} 1 - \sum_{\substack{n\leq \frac{x}{sm}\\ m\mid \phi(smn)\\ s\nmid\phi(smn)\\ (n,m)=1}} 1.
\]
To estimate the second sum, we apply (i) of Lemma \ref{lem:PhiProps}, ignore the conditions involving $m$, and then use \eqref{eq:FundamentalNoDivideUpper} to see that
\[
\sum_{\substack{n\leq \frac{x}{sm}\\ m\mid \phi(smn)\\ s\nmid\phi(smn)\\ (n,m)=1}} 1 \leq \sum_{\substack{n\leq \frac{x}{sm}\\ s\nmid\phi(n)}} 1 \ll \frac{x}{sm} \sum_{p\mid s} \frac{1}{(\log x)^{1/p}}. 
\]
Here we have used that the fact that $\log(x/sm) \asymp \log x$ since $sm\leq D$.  The contribution to $S^{(2)}$ of these $n$ is then
\[
\ll \sum_{\substack{s\leq D\\P(s) \leq \frac{y}{V}}} \sum_{\substack{m\leq \frac{D}{s}\\p(m) > \frac{y}{V} \\ P(m) \leq Uy}} \frac{x}{sm} \sum_{p\mid s} \frac{1}{(\log x)^{1/p}}\ll x\log y \sum_{p\leq \frac{y}{V}} \frac{1}{p(\log x)^{1/p}}.
\]
Since the function $t\mapsto t(\log x)^{1/p}$ is decreasing for $t\leq \log_2 x$, the sum over primes is
\[
\leq \pi\fracp{y}{V} \frac{V}{y(\log x)^{V/y}} \ll \frac{1}{e^V\log y}.
\]
Therefore
\begin{equation}\label{eq:Cleaning4}
S^{(2)}(x) =  S^{(3)}(x) + O\pth{x\exp(-\sqrt{\log_3 x}) + \frac{x}{U} + \frac{x}{e^V}},
\end{equation}
where
\[
S^{(3)}(x) = \sum_{\substack{s\leq D\\P(s) \leq \frac{y}{V}}} \sum_{\substack{m\leq \frac{D}{s}\\p(m) > \frac{y}{V} \\ P(m) \leq Uy}} g(sm)\sum_{\substack{n\leq \frac{x}{sm}\\ m\mid \phi(mn)\\ (n,m)=1}} 1.
\]
Here the divisibility condition should actually be $m\mid \phi(smn)$. However, since $p(m) > P(s)$, we have $(m,s\phi(s)) = 1$, and therefore the two divisibility conditions are equivalent in our ranges of $s$ and $m$.
\subsection{Restriction to Squarefree Moduli}
Next, we remove from $S^{(3)}$ those $m$ which are not squarefree as well as those $m$ for which $(m,n) > 1$. The contribution of those $m$ which are not squarefree is at most
\[
\sum_{\substack{s\leq D\\P(s) \leq \frac{y}{V}}} \sum_{\frac{y}{V} < p\leq Uy} \sum_{\substack{m\leq \frac{D}{sp^2}\\p(m) > \frac{y}{V} \\ P(m) \leq Uy}} \frac{x}{smp^2} \ll \frac{xV}{y} \ll \frac{x}{\sqrt{y}}
\]
by our restriction $V \leq \sqrt{y}$. This error may be absorbed into the error of \eqref{eq:Cleaning4}, and we now consider the sum
\[
\sum_{\substack{s\leq D\\P(s) \leq \frac{y}{V}}} \sumflat_{\substack{m\leq \frac{D}{s}\\p(m) > \frac{y}{V} \\ P(m) \leq Uy}} g(sm)\sum_{\substack{n\leq \frac{x}{sm}\\ m\mid \phi(mn)}} 1 = \sum_{\substack{s\leq D\\P(s) \leq \frac{y}{V}}} \sumflat_{\substack{m\leq \frac{D}{s}\\p(m) > \frac{y}{V} \\ P(m) \leq Uy}} g(sm)\sum_{g\mid m} \sum_{\substack{n\leq \frac{x}{sm}\\ m\mid \phi(mn)\\ (n,m)=g}} 1.
\]
The contribution of $g > 1$ is trivially bounded by
\[
\sum_{\substack{s\leq D\\P(s) \leq \frac{y}{V}}} \sumflat_{\substack{m\leq \frac{D}{s}\\p(m) > \frac{y}{V} \\ P(m) \leq Uy}} \sum_{\substack{g\mid m\\ g> 1}} \frac{x}{smg^2} \ll x\log y \sum_{\substack{g > 1\\ p(g) > \frac{y}{V} \\ P(g) \leq Uy}} \frac{x}{g^2} \ll \frac{xV\log y}{y},
\]
and again this may be absorbed into the error term in \eqref{eq:Cleaning4}. Thus
\begin{equation}\label{eq:Cleaning3}
S^{(3)}(x) =  S^{(4)}(x) + O\pth{x\exp(-\sqrt{\log_3 x}) + \frac{x}{U} + \frac{x}{e^V}},
\end{equation}
where
\[
S^{(4)}(x) = \sum_{\substack{s\leq D\\P(s) \leq \frac{y}{V}}} \sumflat_{\substack{m\leq \frac{D}{s}\\p(m) > \frac{y}{V} \\ P(m) \leq Uy}} g(sm)\sum_{\substack{n\leq \frac{x}{sm}\\ m\mid \phi(mn)\\ (n,m)=1}} 1.
\]

\subsection{Final Cleaning}
Our goal is now to simplify the divisibility condition in the inner sum above. Heuristically, one expects that $(m,\phi(m))$ is the largest divisor of $m$ supported on the primes $\leq \log_2 m$ (see Theorem of 8 of \cite{ErdosLucaPomerance}). For $m$ considered in the sums above, however, we have by the Prime Number Theorem that
\[
\log_2 m \leq \log_2 \sumpth{\prod_{\frac{y}{V} < p \leq Uy}p} \sim \log y + \log(U+V) \ll \log y < \frac{y}{V}
\]
if $y$ is sufficiently large. Thus we expect $(m,\phi(m))=1$ for most $m$. Moreover, if $(m,n\phi(m)) = 1$, then the conditions $m\mid \phi(mn)$ and $m\mid \phi(n)$ are equivalent. As such, we may replace the divisibility condition on $m$ with  $m\mid\phi(n)$ by making an error of at most the contribution from those $m$ with $(m,\phi(m)) > 1$. 

To estimate this error, we note that since $m$ is squarefree, if $(m,\phi(m)) > 1$, then $m$ is divisible by a pair of primes $p,q$ such that $q\equiv 1 \mod{p}$. The contribution of these $m$ to $S^{(4)}(x)$ is bounded by
\[
\begin{aligned}
x \sum_{\substack{s\leq D\\ P(s) \leq \frac{y}{V}}} \frac{1}{s} \sum_{\frac{y}{V} < p \leq Uy} \sum_{\substack{q \leq y\\ q\equiv 1 \mod{p}}} \sumflat_{\substack{m\leq \frac{D}{s}\\ P(m) \leq Uy\\ p(m) > \frac{y}{V}\\ pq\mid m}} \frac{1}{m} &\ll x \log y \sum_{\frac{y}{V} < p \leq Uy} \frac{1}{p} \sum_{\substack{q \leq Uy\\ q\equiv 1 \mod{p}}} \frac{1}{q} \\
&\ll x\log y\log_2 y \sum_{p > \frac{y}{V}} \frac{1}{p^2}\\
&\ll \frac{xV\log_2 y}{y}.
\end{aligned}
\]
This error may be absorbed into the error of \eqref{eq:Cleaning4}. The condition $(n,m) = 1$ may now be removed at the cost of the same error we made when introducing it, and so
\begin{equation}\label{eq:Cleaning5}
S^{(4)}(x) =  S^{(5)}(x) + O\pth{x\exp(-\sqrt{\log_3 x}) + \frac{x}{U} + \frac{x}{e^V}},
\end{equation}
where
\[
S^{(5)}(x) = \sum_{\substack{s\leq D\\P(s) \leq \frac{y}{V}}} \sumflat_{\substack{m\leq \frac{D}{s}\\p(m) > \frac{y}{V} \\ P(m) \leq Uy}} g(sm)\sum_{\substack{n\leq \frac{x}{sm}\\ m\mid \phi(n)}} 1.
\]

\subsection{Final Truncation}\label{sec:FinalTruncation}
We aim to apply Proposition \ref{prop:FundamentalNoDivide} in the form of Corollary \ref{cor:NoDivide}. To do so, the variable $m$ must have a bounded number of prime factors. Let $K\geq 1$ be a fixed positive integer. We perform one last truncation and restrict our attention to those $m$ with $\omega(m) \leq K$. We write
\[
S^{(5)}(x) = S^{(6)}(x) + E(x),
\]
where
\begin{equation}\label{eq:Cleaning6}
S^{(6)}(x) = \sum_{\substack{s\leq D\\ P(s) \leq \frac{y}{V}}} \sumflat_{\substack{m\leq \frac{D}{s}\\ P(m) \leq y\\ p(m) > \frac{y}{V}\\ \omega(m)\leq K}} g(sm) \sum_{\substack{n\leq \frac{x}{sm}\\ m\mid \phi(n)}} 1
\end{equation}
and $E(x) = S^{(5)}(x) - S^{(6)}(x)$. To estimate $E(x)$, we bound the sums over $n$ trivially and use the binomial theorem to see that
\[
\abs{E(x)} \leq \sum_{\substack{s\leq D\\ P(s) \leq \frac{y}{V}}} \sumflat_{\substack{m\leq \frac{D}{s}\\ P(m) \leq Uy\\ p(m) > \frac{y}{V}\\\omega(m)\geq K+1}} \frac{x}{sm} \leq x \sum_{\substack{s\geq 1\\ P(s) \leq \frac{y}{V}}} \frac{1}{s} \sum_{k\geq K+1} \sumflat_{\substack{m > 1\\ P(m) \leq Uy\\ p(m) > \frac{y}{V}\\ \omega(m) = k}} \frac{1}{m} \ll x\log y \sum_{k\geq K+1} \frac{1}{k!} \sumpth{\sum_{\frac{y}{V}<p \leq Uy} \frac{1}{p}}^k.
\]
Using \eqref{eq:MertensSum} and the fact that $U,V\leq \sqrt{y}$, the sum over primes is
\[
\log\fracp{\log(Uy)}{\log(y/V)} + O\pth{\exp(-C\sqrt{\log(y/V)})} \leq \calc\frac{\log U+\log V}{\log y}.
\]
for some positive absolute constant $\calc$. Therefore
\[
E(x)\ll x \log y \sum_{k\geq K+1} \frac{\calc^K}{k!} \fracp{\log U+\log V}{\log y}^k \ll_K  x \log y\fracp{\log U+\log V}{\log y}^{K+1}.
\]
Combining all of our estimates in this section, we find that
\begin{equation}\label{eq:FullyCleaned}
S(x) = S^{(6)}(x) + O_K\pth{x\exp(-\sqrt{\log_3 x}) + \frac{x}{U} + \frac{x}{e^V} + \log y\fracp{\log U+\log V}{\log y}^{K+1} },
\end{equation}
where $S^{(6)}$ is given by \eqref{eq:Cleaning6}.

\section{Evaluation of $S^{(6)}(x)$}\label{sec:FinalEval}

We now apply Proposition \ref{prop:FundamentalNoDivide} in the form of Corollary \ref{cor:NoDivide} to the inner sum over $n$ in $S^{(6)}(x)$, and we have
\[
\begin{aligned}
S^{(6)}(x) &= x\sum_{\substack{s\leq D\\ P(s) \leq \frac{y}{V}}} \sumflat_{\substack{m\leq \frac{D}{s}\\ P(m) \leq Uy\\ p(m) > \frac{y}{V}\\\omega(m)\leq K}} \frac{g(sm)}{sm} \prod_{p\mid m} \pth{1-\frac{1}{(\log(x/sm))^{1/(p-1)}}} \\
&\qquad + O_K\sumpth{D+\sum_{\substack{s\leq D\\ P(s) \leq \frac{y}{V}}} \sumflat_{\substack{1<m\leq \frac{D}{s}\\ P(m) \leq Uy\\ p(m) > \frac{y}{V}\\\omega(m)\leq K}}\frac{x}{sm}\pth{\frac{\log p(m)}{p(m)} + \frac{\log_2 x}{p(m)^2}}}.
\end{aligned}
\]
Here the error term $O(D)$ arises from the contribution of $m=1$. The error term contributes at most
\begin{equation}\label{eq:L71error}
\pth{\frac{V\log y}{y} + \frac{V^2\log_2 x}{y^2}} \sum_{\substack{m\geq 1\\ P(m) \leq Uy}} \frac{x}{m} \ll x\frac{V\log y(\log y+V)}{y}.
\end{equation}
To evaluate the main term, we would like $\log x$ in the product rather than $\log(x/sm)$. Note that
\[
1-\frac{1}{(\log(x/sm))^{1/(p-1)}} = \pth{1-\frac{1}{(\log x)^{1/(p-1)}}}\pth{1+O\fracp{\log D}{p\log x}}
\]
and also
\[
\prod_{p\mid m}\pth{1+O\fracp{\log D}{p\log x}} = 1 + O\fracp{K \log D}{p(m)\log x}.
\]
The contribution of the error term above may be absorbed into \eqref{eq:L71error}. Next, we extend the sums over $s$ and $m$ to all integers. Using Rankin's trick in the form of \eqref{eq:Rankin2}, the contribution of $sm > D$ is bounded by
\[
x\sum_{\substack{n > D\\ P(n) \leq Uy}} \frac{1}{n} \ll \frac{x(\log y)^3}{D^{\frac{1}{\log(Uy)}}} \ll \frac{x}{(\log_2 x)^{\frac{1}{4}}}.
\] 
The error term $O(D)$ may be absorbed into this term, and so
\[
\begin{aligned}
S^{(6)}(x) &= x\sum_{\substack{s\geq 1\\ P(s) \leq \frac{y}{V}}} \sumflat_{\substack{m\geq 1\\ P(m) \leq Uy\\ p(m) > \frac{y}{V}\\\omega(m)\leq K}} \frac{g(sm)}{sm} \prod_{p\mid m} \pth{1-\frac{1}{(\log x)^{1/(p-1)}}} \\
&\qquad + O_K\sumpth{\frac{x}{(\log_2 x)^{\frac{1}{4}}} + x\frac{V\log y(\log y+V)}{y}}.
\end{aligned}
\]
The condition $\omega(d) \leq K$ may now be removed with the same error as when we introduced it, and the main term of $S^{(6)}(x)$ may now be written as the product
\begin{equation}\label{eq:C6Main}
x \prod_{p\leq \frac{y}{V}} \sumpth{\sum_{j\geq 0} \frac{g(p^j)}{p^j}}  \prod_{\frac{y}{V} < p\leq Uy} \pth{1+\frac{g(p)}{p} \pth{1-\frac{1}{(\log x)^{1/(p-1)}}}}.
\end{equation}
We simplify the second product slightly by noting that
\[
(\log x)^{-1/(p-1)} = (\log x)^{-1/p} \pth{1+O\fracp{\log_2 x}{p^2}} = (\log x)^{-1/p} + O\fracp{\log_2 x}{p^2}.
\]
In the range of $p$ in the second product, we then have
\[
1+\frac{g(p)}{p} \pth{1-\frac{1}{(\log x)^{1/(p-1)}}} = \pth{1+\frac{g(p)}{p} \pth{1-\frac{1}{(\log x)^{1/p}}}}\pth{1 + O\fracp{V}{p^2}}.
\]
Then since
\[
\prod_{\frac{y}{V} < p\leq Uy}\pth{1+O\fracp{V}{p^2}} = 1 + O\fracp{V^2}{y\log y}, 
\]
the product in \eqref{eq:C6Main} is
\[
x \prod_{p\leq \frac{y}{V}} \sumpth{\sum_{j\geq 0} \frac{g(p^j)}{p^j}}  \prod_{\frac{y}{V} < p\leq Uy}  \pth{1+\frac{g(p)}{p} \pth{1-\frac{1}{(\log x)^{1/p}}}} \pth{1 + O\fracp{V^2}{y\log y}}.
\]
The contribution of the error term is
\[
\ll \frac{V^2}{y\log y}\prod_{p\leq Uy} \pth{1-\frac{1}{p}}^{-1} \ll \frac{V^2}{y},
\]
and this may be absorbed into \eqref{eq:L71error}. Thus
\[
\begin{aligned}
S^{(6)}(x) &= x \prod_{p\leq \frac{y}{V}} \sumpth{\sum_{j\geq 0} \frac{g(p^j)}{p^j}}  \prod_{\frac{y}{V} < p\leq Uy}  \pth{1+\frac{g(p)}{p} \pth{1-\frac{1}{(\log x)^{1/p}}}} \\
&\qquad+ O_K\sumpth{\frac{x}{(\log_2 x)^{\frac{1}{4}}} + x\frac{V\log y(\log y+V)}{y} + \log y\fracp{\log U+\log V}{\log y}^{K+1}}.
\end{aligned}
\]

Inserting the above calculations into \eqref{eq:FullyCleaned}, we find that
\begin{equation}\label{eq:MainTermProduct}
S(x) =x (\calp(x) + \cale(x)),
\end{equation}
where
\[
\calp(x) = \prod_{p\leq \frac{y}{V}} \sumpth{\sum_{j\geq 0} \frac{g(p^j)}{p^j}}  \prod_{\frac{y}{V} < p\leq Uy}  \pth{1+\frac{g(p)}{p} \pth{1-\frac{1}{(\log x)^{1/p}}}}
\]
and
\begin{equation}\label{eq:ColeError}
\abs{\cale(x)} \ll_K \exp(-\half\sqrt{\log_3 x}) + \frac{1}{U} + \frac{1}{e^V} + \log y\fracp{\log U+\log V}{\log y}^{K+1}  + \frac{V\log y(\log y+V)}{y}. 
\end{equation}
\section{Conclusion}\label{sec:Conclusion}
It remains to evaluate the product $\calp(x)$. To do so, we first write
\[
\calp(x) = \calm(x) \calp_1(x) \calp_2(x),
\]
where
\[
\begin{aligned}
\calm(x) &= \prod_{p\leq y} \sumpth{\sum_{j\geq 0} \frac{g(p^j)}{p^j}} \\
\calp_1(x) &= \prod_{\frac{y}{V} < p\leq y} \sumpth{\sum_{j\geq 0} \frac{g(p^j)}{p^j}}^{-1} \pth{1+\frac{g(p)}{p} \pth{1-\frac{1}{(\log x)^{1/p}}}} \\
\calp_2(x) &=  \prod_{y < p\leq Uy}  \pth{1+\frac{g(p)}{p} \pth{1-\frac{1}{(\log x)^{1/p}}}}.
\end{aligned}
\]
Note that
\[
\sumpth{\sum_{j\geq 0} \frac{g(p^j)}{p^j}}^{-1} = 1 - \frac{g(p)}{p} + O\fracp{1}{p^2},
\]
and so
\[
\begin{aligned}
\sumpth{\sum_{j\geq 0} \frac{g(p^j)}{p^j}}^{-1}\pth{1+\frac{g(p)}{p} \pth{1-\frac{1}{(\log x)^{1/p}}}} &= 1 - \frac{g(p)}{p(\log x)^{1/p}} + O\fracp{1}{p^2} \\
&= \pth{1 - \frac{g(p)}{p(\log x)^{1/p}}}\pth{1+O\fracp{1}{p^2}}.
\end{aligned}
\]
Thus
\[
\calp(x) = \calm(x) \calp_1^*(x) \calp_2(x) + O\fracp{V}{y},
\]
where
\[
\calp_1^*(x) = \prod_{\frac{y}{V} < p \leq y} \pth{1 - \frac{g(p)}{p(\log x)^{1/p}}}.
\]
We further investigate the products $\calp_1^*$ and $\calp_2$ by noting that if $p\leq y$, then
\[
1+\frac{g(p)}{p} \pth{1-\frac{1}{(\log x)^{1/p}}} = \exp\pth{\frac{g(p)}{p} \pth{1-\frac{1}{(\log x)^{1/p}}} + O\fracp{1}{p^2}},
\]
and if $p > y$, then
\[
1 - \frac{g(p)}{p(\log x)^{1/p}} = \exp\pth{- \frac{g(p)}{p(\log x)^{1/p}}}\pth{1 + O\fracp{1}{p^2}}.
\]
Thus
\[
\begin{aligned}
\calp_1^*(x)\calp_2(x) &= \exp\sumpth{- \sum_{\frac{y}{V} < p \leq y} \frac{g(p)}{p(\log x)^{1/p}}+ \sum_{y< p \leq Uy} \frac{g(p)}{p} \pth{1-\frac{1}{(\log x)^{1/p}}} } \prod_{\frac{y}{V} < p\leq Uy} \pth{1+O\fracp{1}{p^2}} \\
&=  \exp\sumpth{- \sum_{\frac{y}{V} < p \leq y} \frac{g(p)}{p(\log x)^{1/p}}+ \sum_{y< p \leq Uy} \frac{g(p)}{p} \pth{1-\frac{1}{(\log x)^{1/p}}} } \pth{1+O\fracp{V}{y\log y}}.
\end{aligned}
\]
Arguing as in the proof of Corollary \ref{cor:asymptotic}, we have
\[
\sumabs{\sum_{\frac{y}{V} < p \leq y} \frac{g(p)}{p} \pth{1-\frac{1}{(\log x)^{1/p}}} - \sum_{y < p \leq Uy} \frac{g(p)}{p(\log x)^{1/p}}} \ll \frac{1}{\log_3 x}.
\]
The contribution of the error term may thus be absorbed into $\cale(x)$ in \eqref{eq:MainTermProduct}, and so
\[
S(x) = x \sumpth{\calm(x) \exp\sumpth{- \sum_{\frac{y}{V} < p \leq y} \frac{g(p)}{p(\log x)^{1/p}}+ \sum_{y< p \leq Uy} \frac{g(p)}{p} \pth{1-\frac{1}{(\log x)^{1/p}}} } + \cale(x)},
\]
where $\cale(x)$ satisfies \eqref{eq:ColeError}. To complete the proof of Theorem \ref{thm:Main}, we note that
\begin{equation}\label{eq:smallPrimes}
\sumabs{\sum_{p\leq \frac{y}{V}} \frac{g(p)}{p(\log x)^{1/p}}} \ll \frac{1}{e^V \log y}
\end{equation}
and
\begin{equation}\label{eq:largePrimes}
\sumabs{\sum_{p > Uy} \frac{g(p)}{p} \pth{1-\frac{1}{(\log x)^{1/p}}}} \ll \frac{1}{U\log y}.
\end{equation}
As noted in the introduction, $\calm(x) \ll \log y$, and so the contribution of the sums above may be absorbed into $\cale(x)$. We have
\[
\abs{\cale(x)} \ll_K \exp(-\half\sqrt{\log_3 x}) + \frac{1}{U} + \frac{1}{e^V} + \log y\fracp{\log U+\log V}{\log y}^{K+1}  + \frac{V\log y(\log y+V)}{y},
\]
and we conclude the proof of Theorem \ref{thm:Main} by choosing 
\begin{equation}\label{eq:chooseUV}
U = (\log_3 x)^{K}, \qquad V = \log_3 x,
\end{equation}
and so
\[
\abs{\cale(x)} \ll_K x \frac{(\log_4 x)^{K+1}}{(\log_3 x)^K} \ll x \frac{1}{(\log_3 x)^{K-1/2}},
\]
say. Since $K$ was arbitrary, this gives Theorem \ref{thm:Main}.

\section{Evaluation of Prime Sums}\label{sec:PrimeSums}

We now consider the expression given in Theorem \ref{thm:Main} for a variety of arithmetic functions $g$. First, we need the following consequence of the Prime Number Theorem with classical error term. Throughout this section, we regard $K$ is a fixed positive integer governing the precision of our asymptotic formulas.

\begin{lem}\label{lem:PrimeSum}
For any integer $K\geq 1$, we have
\[
\begin{aligned}
\sum_{p\leq \log_2 x} \frac{1}{p(\log x)^{1/p}} &= \frac{1}{\log_3 x} \sum_{k=0}^K \frac{a_k}{(\log_3 x)^k} + O_K\fracp{1}{(\log_3 x)^{K+2}},\\
\sum_{p > \log_2 x} \frac{1}{p} \pth{1-\frac{1}{(\log x)^{1/p}}} &= \frac{1}{\log_3 x} \sum_{k=0}^K \frac{b_k}{(\log_3 x)^k} + O_K\fracp{1}{(\log_3 x)^{K+2}},\\
\end{aligned}
\]
where
\[
\begin{aligned}
a_k &= \int_{1}^{\infty} \frac{(\log t)^k}{te^t}\, dt, \\
b_k &= \int_{1}^{\infty} \frac{(-\log t)^k}{t}\pth{1-\frac{1}{e^{1/t}}}\, dt.
\end{aligned}
\]
\end{lem}

\begin{proof}
As in the proof of Theorem \ref{thm:Main}, we set $y=\log_2 x$ and choose $U$ and $V$ as in \eqref{eq:chooseUV}, except with $K+2$ in place of $K$ in the choice of $U$. Then \eqref{eq:smallPrimes} and \eqref{eq:largePrimes} give
\[
\sum_{p\leq \frac{y}{V}} \frac{1}{p(\log x)^{1/p}} + \sum_{p > U\log_2 x} \frac{1}{p} \pth{1-\frac{1}{(\log x)^{1/p}}} \ll \frac{1}{(\log_3 x)^{K+2}}.
\]
To evaluate the remaining sums, we use the Prime Number Theorem in the form
\[
\pi(t) = \int_{2}^{t} \frac{dz}{\log z} + R(t),
\]
where
\[
R(t) \ll t \exp(-C\sqrt{\log t}).
\]
Integrating by parts, we have
\[
\begin{aligned}
\sum_{\frac{y}{V} < p\leq y} \frac{1}{pe^{y/p}} &= \int_{\frac{y}{V}}^{y} \frac{1}{t(\log t)e^{y/t}}\, dt + O\pth{\exp(-C\sqrt{\log y})}, \\
\sum_{y < p \leq Uy} \frac{1}{p}\pth{1-\frac{1}{pe^{y/p}}} &=  \int_{y}^{Uy} \frac{1}{t(\log t)}\pth{1-\frac{1}{e^{y/t}}}\, dt + O\pth{\exp(-C\sqrt{\log y})}, \\
\end{aligned}
\]
where the value of $C$ is now slightly smaller. Changing variables $t\to \frac{y}{t}$ in the first integral and $t\to yt$ in the second yields
\[
\begin{aligned}
\int_{\frac{y}{V}}^{y} \frac{1}{t(\log t)e^{y/t}}\, dt &= \frac{1}{\log y}  \int_{1}^{V} \frac{1}{te^t} \pth{1-\frac{\log t}{\log y}}^{-1} dt, \\
\int_{y}^{Uy} \frac{1}{t(\log t)}\pth{1-\frac{1}{e^{y/t}}}\, dt &= \frac{1}{\log y} \int_{1}^{U} \frac{1}{t} \pth{1-\frac{1}{e^{1/t}}} \pth{1+\frac{\log t}{\log y}}^{-1} dt.
\end{aligned}
\]
If $\abs{z} \leq \frac{1}{2}$, we have
\[
\frac{1}{1-z} = \sum_{k=0}^K z^k + O\fracp{1}{z^{K+1}}
\]
uniformly in $K$. Since $t\leq V\leq U \leq \sqrt{y}$ for $x$ sufficiently large, we have
\[
\begin{aligned}
\int_{1}^{V} \frac{1}{te^t} \pth{1-\frac{\log t}{\log y}}^{-1} dt &= \sum_{k=0}^K \frac{1}{(\log y)^k} \int_{1}^{V} \frac{(\log t)^k}{te^t}\, dt + O_K\fracp{1}{(\log y)^{K+1}}, \\
\int_{1}^{U} \frac{1}{t} \pth{1-\frac{1}{e^{1/t}}} \pth{1+\frac{\log t}{\log y}}^{-1} dt &= \sum_{k=0}^K \frac{(-1)^k}{(\log y)^k} \int_{1}^{U} \frac{(\log t)^k}{t}\pth{1-\frac{1}{e^{1/t}}} dt + O_K\fracp{1}{(\log y)^{K+1}}.
\end{aligned}
\]
The tails of the integrals are
\[
\int_{V}^{\infty} \frac{(\log t)^k}{te^t}\, dt \ll_K \frac{(\log V)^k}{Ve^V}, \qquad\int_{U}^{\infty} \frac{(\log t)^k}{t}\pth{1-\frac{1}{e^{1/t}}}\, dt \ll \frac{(\log U)^k}{U}.
\]
Combining the above calculations now gives the claimed asymptotic expansions.

\end{proof}

\noindent \textbf{Remark.} The coefficients $a_k$ and $b_k$ actually arise from the Laurent series of $\Gamma(s)$ around $s=0$, where $\Gamma$ is the usual gamma function. To see this, let
\[
F(s) = \int_{1}^{\infty} \pth{1-\frac{1}{e^{1/t}}}\frac{t^{-s}}{t}\, dt -  \int_{1}^{\infty}\frac{t^s}{te^t}\, dt
\]
and observe that $b_k - a_k = F^{(k)}(0)$, where here $F^{(k)}$ denotes the $k$th derivative of $F$. Changing variables $t\to t^{-1}$ in the first integral, we have
\[
F(s) = \int_{0}^{1} \pth{1-\frac{1}{e^t}} \frac{t^s}{t}\, dt -  \int_{1}^{\infty}\frac{t^s}{te^t}\, dt = \frac{1}{s} - \Gamma(s).
\]

\subsection{Asymptotics for Indicator Functions}

Let $\calb$ be a set of integers whose indicator function $f$ is multiplicative. Then, as stated in the introduction, we have
\[
f(n) = \sum_{d\mid n} g(d)
\]
for a multiplicative function $g$ with $\abs{g(n)} \leq 1$ for all $n$. More specifically, we have $g(p^j) = f(p^j) - f(p^{j-1})$ for all $j\geq 1$,  and in particular, $g(p) = f(p) - 1$. Thus
\begin{equation}\label{eq:QIndicator}
\begin{aligned}
	\calq_g(x) &= -\sum_{p\leq \log_2 x} \frac{f(p)-1}{p(\log x)^{1/p}} + \sum_{p > \log_2 x} \frac{f(p)-1}{p} \pth{1-\frac{1}{(\log x)^{1/p}}} \\
	&= \calq_f(x) + \frac{1}{\log_3 x} \sum_{k=0}^K \frac{a_k-b_k}{(\log_3 x)^k} + O_K\fracp{1}{(\log_3 x)^{K+2}}
\end{aligned}
\end{equation}
by Lemma \ref{lem:PrimeSum}. Likewise
\[
\begin{aligned}
\prod_{p\leq \log_2 x} \sumpth{\sum_{j\geq 0} \frac{g(p^j)}{p^j}} &= \prod_{p\leq \log_2 x} \pth{1-\frac{1}{p}} \sumpth{\sum_{j\geq 0} \frac{f(p^j)}{p^j}} \\
&= \frac{1}{e^\gamma \log_3 x} \prod_{p\leq \log_2 x} \sumpth{\sum_{j\geq 0} \frac{f(p^j)}{p^j}} + O\pth{\exp(-C\sqrt{\log_3 x})}
\end{aligned}
\]
by \eqref{eq:MertensSum}. Thus, for certain coefficients $c_k$, we have
\begin{equation}\label{eq:AsymptoticIndicator}
\sum_{\substack{n\leq x\\ (n,\phi(n)) \in \calb}} 1 = \calf(x)\frac{x}{e^\gamma \log_3 x} \sumpth{1+ \sum_{k=1}^K \frac{c_k}{(\log_3 x)^k}} + O_K\fracp{1}{(\log_3 x)^{K+2}},
\end{equation}
where
\[
\calf(x) = \sumpth{\prod_{p\leq \log_2 x} \sum_{j\geq 0} \frac{f(p^j)}{p^j}} \exp(\calq_f(x)).
\]
In fact, the coefficients $c_k$ are exactly the coefficients in \eqref{eq:PollackPoincare}. We now specialize $f$ to certain sets $\calb$. 

\subsubsection{$r$th Powers}
If $\calb$ is the set of $r$th powers for some $r\geq 2$, then $f(p) = 0$ for all $p$ and $\calq_f(x) = 0$. Likewise
\[
\prod_{p\leq \log_2 x} \sum_{j\geq 0} \frac{f(p^j)}{p^j} = \prod_{p\leq \log_2 x} \pth{1-\frac{1}{p^r}}^{-1} = \zeta(r) + O\fracp{1}{(\log_2 x)^{r-1}\log_3 x},
\]
and so \eqref{eq:AsymptoticIndicator} gives
\begin{equation}\label{eq:rthPowers}
	\sum_{\substack{n\leq x\\ (n,\phi(n)) = m^r}} 1 = \frac{\zeta(r)x}{e^\gamma \log_3 x}\sumpth{1+ \sum_{k=1}^K \frac{c_k}{(\log_3 x)^k} } + O_K\fracp{x}{(\log_3 x)^{K+2}}.
\end{equation}
The above argument also works if $r=0$ upon noting that $\calf(x) = 1$ in this case, and so we recover \eqref{eq:PollackPoincare} by a different proof than that originally given by Pollack in \cite{PollackErdos}.

\subsubsection{$r$-free Integers}
If $\calb$ is the set of integers free of $r$th powers of primes, then $f(p^j) = 0$ for $j \geq r$. In particular, $f(p) = 1$, and so $\calq_g(x) = 0$ by \eqref{eq:QIndicator}. Thus
\[
\begin{aligned}
\prod_{p\leq \log_2 x} \sum_{j\geq 0} \frac{f(p^j)}{p^j} = \prod_{p\leq \log_2 x} \sumpth{\sum_{j\leq r-1} \frac{1}{p^j}} &=  \prod_{p\leq \log_2 x} \pth{1-\frac{1}{p}}^{-1} \pth{1-\frac{1}{p^r}} \\
&= \frac{e^\gamma x\log_3 x}{\zeta(r)} + O\pth{\exp(-C\sqrt{\log_3 x})}
\end{aligned}
\]
by \eqref{eq:MertensProduct}. Therefore
\begin{equation}\label{eq:rfree}
\sum_{\substack{n\leq x\\ (n,\phi(n))\ \text{is $r$-free}}} 1 = \frac{x}{\zeta(r)} + O_K\fracp{x}{(\log_3 x)^{K+2}}.
\end{equation}

\subsubsection{Sums of Squares}
If $\calb$ is the set of integers representable as a sum of two squares, then
\[
f(p^j) = \begin{cases}
1 & \text{if $p = 2$ or $p\equiv 1 \mod{4}$}, \\
1 & \text{if $p \equiv 3 \mod{4}$ and $j$ is even}, \\
0 & \text{if $p \equiv 3 \mod{4}$ and $j$ is odd}.
\end{cases}
\]
Thus
\[
\prod_{p\leq \log_2 x} \sum_{j\geq 0} \frac{f(p^j)}{p^j} = 2\prod_{\substack{p\leq \log_2 x\\ p\equiv 1\mod{4}}} \pth{1-\frac{1}{p}}^{-1}  \prod_{\substack{p\leq \log_2 x\\ p\equiv 3\mod{4}}} \pth{1-\frac{1}{p^2}}^{-1}. 
\]
We rewrite the square of the sum as
\[
\begin{aligned}
\sumpth{\prod_{p\leq \log_2 x} \sum_{j\geq 0} \frac{f(p^j)}{p^j}}^2 &= 2 \prod_{p\leq \log_2 x} \pth{1-\frac{1}{p}}^{-1} \prod_{p\leq \log_2 x} \pth{1-\frac{\chi_4(p)}{p}}^{-1} \prod_{\substack{p\leq \log_2 x\\ p\equiv 3\mod{4}}} \pth{1-\frac{1}{p^2}}^{-1} \\
&= \frac{\pi e^\gamma\log_3 x}{2} \prod_{p\equiv 3 \mod{4}} \pth{1-\frac{1}{p^2}}^{-1} \pth{1+O\pth{\exp(-C\sqrt{\log_3 x})}},
\end{aligned}
\]
where $\chi_4$ is the nontrivial Dirichlet character mod 4. Thus
\[
\prod_{p\leq \log_2 x} \sumpth{\sum_{j\geq 0} \frac{g(p^j)}{p^j}} =  \frac{\sqrt{\pi}B}{\sqrt{e^\gamma\log_3 x}}\pth{1+O\pth{\exp(-C\sqrt{\log_3 x})}},
\]
where $B$ is the usual Landau-Ramanujan constant given by 
\[
B^2 = \frac{1}{2} \prod_{p\equiv 3 \mod{4}} \pth{1-\frac{1}{p^2}}^{-1}.
\]
To evaluate $\calq_g(x)$, we note that 
\[
\calq_g(x) = \sum_{\substack{p\leq \log_2 x\\p\equiv 3 \mod{4}}} \frac{1}{p(\log x)^{1/p}} - \sum_{\substack{p > \log_2 x\\ p \equiv 3\mod{4}}} \frac{1}{p} \pth{1-\frac{1}{(\log x)^{1/p}}}.
\]
The prime number theorem for primes $p\equiv 3 \mod{4}$ states that for $x$ sufficiently large,
\[
\pi(x;3,4) = \frac{1}{2} \int_{2}^{x} \frac{dt}{\log t} + O\pth{t\exp(-C\sqrt{\log t})}.
\]
Arguing as in the proof of Lemma \ref{lem:PrimeSum}, we have
\[
\begin{aligned}
\sum_{\substack{p\leq \log_2 x\\ p\equiv 3 \mod{4}}} \frac{1}{p(\log x)^{1/p}} &= \frac{1}{2\log_3 x} \sum_{k=0}^K \frac{a_k}{(\log_3 x)^k} + O_K\fracp{1}{(\log_3 x)^{K+2}},\\
\sum_{\substack{p > \log_2 x\\ p \equiv 3\mod{4}}} \frac{1}{p} \pth{1-\frac{1}{(\log x)^{1/p}}} &= \frac{1}{2\log_3 x} \sum_{k=0}^K \frac{b_k}{(\log_3 x)^k} + O_K\fracp{1}{(\log_3 x)^{K+2}}.
\end{aligned}
\]
Combining all of the estimates above, we find that there are constants $d_k$ such that
\begin{equation}\label{eq:sumofsquares}
\sum_{\substack{n\leq x\\ (n,\phi(n)) \in \calb}} 1 = \frac{\sqrt{\pi}Bx}{\sqrt{e^\gamma\log_3 x}} \sumpth{1+ \sum_{k=1}^K \frac{d_k}{(\log_3 x)^k}} + O_K\fracp{1}{(\log_3 x)^{K+\frac{3}{2}}}.
\end{equation}

\subsection{The Average Number of Divisors of $(n,\phi(n))$}

We conclude by studying the average number of divisors of $(n,\phi(n))$, which was the original motivation for this work. In this case, we have $f(n) = \tau(n)$, so $g(n) = 1$. Combining Theorem \ref{thm:Main} and Lemma \ref{lem:PrimeSum}, we have
\[
\sum_{n\leq x} \tau((n,\phi(n))) = e^\gamma x \log_3 x \exp\sumpth{\frac{1}{\log_3 x} \sum_{k=0}^K \frac{b_k-a_k}{(\log_3 x)^k}}+O_K\fracp{x}{(\log_3 x)^{K}}.
\]
In particular, there exist constants $c_k'$ such that
\begin{equation}\label{eq:Divisors}
\sum_{n\leq x} \tau((n,\phi(n))) = e^\gamma x \log_3 x \sumpth{1+ \sum_{k=1}^K \frac{c_k'}{(\log_3 x)^k}} + O_K\fracp{x}{(\log_3 x)^{K}}.
\end{equation}

\bibliography{references}

\begin{thebibliography}{1}

\bibitem{ErdosLucaPomerance}
P.~Erd\H{o}s, F.~Luca, and C.~Pomerance.
\newblock On the proportion of numbers coprime to a given integer.
\newblock In {\em Anatomy of integers}, volume~46 of {\em CRM Proc. Lecture
  Notes}, pages 47--64. Amer. Math. Soc., Providence, RI, 2008.

\bibitem{ErdosGroup}
P.~Erd\"{o}s.
\newblock Some asymptotic formulas in number theory.
\newblock {\em J. Indian Math. Soc. (N.S.)}, 12:75--78, 1948.

\bibitem{HalberstamRichert}
H.~Halberstam and H.-E. Richert.
\newblock {\em Sieve methods}.
\newblock London Mathematical Society Monographs, No. 4. Academic Press
  [Harcourt Brace Jovanovich, Publishers], London-New York, 1974.

\bibitem{Koukoulopolus}
D.~Koukoulopoulos.
\newblock {\em The distribution of prime numbers}, volume 203 of {\em Graduate
  Studies in Mathematics}.
\newblock American Mathematical Society, Providence, RI, [2019] \copyright
  2019.

\bibitem{LPR}
N.~Lebowitz-Lockard, P.~Pollack, and A.~S. Roy.
\newblock {Distribution mod p of Euler’s Totient and the Sum of Proper
  Divisors}.
\newblock {\em Michigan Mathematical Journal}, pages 1 -- 24, 2023.

\bibitem{PollackOrders}
P.~Pollack.
\newblock Numbers which are orders only of cyclic groups.
\newblock {\em Proc. Amer. Math. Soc.}, 150(2):515--524, 2022.

\bibitem{PollackErdos}
P.~Pollack.
\newblock Numbers which are orders only of cyclic groups.
\newblock {\em Proc. Amer. Math. Soc.}, 150(2):515--524, 2022.

\bibitem{Pomerance}
C.~Pomerance.
\newblock On the distribution of amicable numbers.
\newblock {\em J. Reine Angew. Math.}, 293(294):217--222, 1977.

\bibitem{Szele}
T.~Szele.
\newblock \"{U}ber die endichen {O}rdnungszahlen, zu denen nur eine {G}ruppe
  geh\"{o}rt.
\newblock {\em Comment. Math. Helv.}, 20:265--267, 1947.

\end{thebibliography}
\end{document}